\documentclass[10pt]{amsart}       
\usepackage{amsmath}
\usepackage{amssymb}
\usepackage{amsmath}
\usepackage{mathtools}
\let\ge=\geqslant
\let\le=\leqslant
\DeclareMathOperator*{\esssup}{ess\,sup}

\newcommand{\REM}[1]{\relax}

\theoremstyle{theorem}

\newtheorem{result}{Theorem}
\newtheorem{theorem}{Theorem}
\newtheorem{lemma}{Lemma}

\theoremstyle{definition}
\newtheorem{definition}{Definition}
\newtheorem{example}{Example}

\theoremstyle{remark}
\newtheorem{remark}{Remark}
\numberwithin{equation}{section}

\newcommand{\Hol}{{\sf Hol}}

\newcommand{\Poiss}{\mathcal P}

\newcommand{\UH}{\mathbb{H}}

\newcommand{\Real}{\mathbb{R}}

\newcommand{\Complex}{\mathbb{C}}
\newcommand{\ComplexE}{\overline{\mathbb{C}}}

\newcommand{\UD}{\mathbb{D}}
\newcommand{\UC}{{\partial\UD}}
\newcommand{\clS}{\mathcal{S}}

\newcommand{\Maponto}
{\xrightarrow{\scriptstyle \!\mathsf{onto}\,}}

\newcommand{\Mapinto}
{\xrightarrow{\hbox{\lower.2ex\hbox{$\scriptstyle \smash{\mathsf{into}}$}}\,}}

\let\R=\Real

\let\C=\Complex

\newcommand{\proofbox}{\hfill$\Box$}

\newcommand{\STOP}{\par\hbox to\textwidth{\color{red}\leaders\hbox{\,STOP\,}\hfil}\par}

\newcommand{\mcite}[1]{\csname b@#1\endcsname}

\def\id{{\sf id}}

\def\Re{\mathop{\mathsf{Re}}}
\def\Im{\mathop{\mathtt{I}\hskip-.1ex\mathsf{m}}}

\newcommand{\di}{\mathrm{d}}

\newcommand{\SHOWCORRECTIONS}{%
\newcommand{\nv}[1]{{\color{green!60!black}##1}}%
\newcommand{\comm}[1]{{\color[rgb]{0.5,0,0.5}##1}}%
\newcommand{\dv}[1]{{\color[rgb]{0.65,0.74,0.79}\sout{##1}}}%
\newcommand{\IN}[1]{{\color[rgb]{1.00,0.33,0.33}##1}}
\newcommand{\ID}[1]{{\color[rgb]{0.65,0.55,0.62}\sout{##1}}}%
\newcommand{\IC}[1]{{\color[rgb]{0.00,0.00,1}##1}}%
\newcommand{\RD}[1]{\textcolor{red}{##1}}%
}%

\newcommand{\HIDECORRECTIONS}{%
\newcommand{\nv}[1]{##1}
\newcommand{\dv}[1]{\relax}%
\newcommand{\comm}[1]{\relax}%
\let\IN=\nv%
\let\ID=\dv%
\let\IC=\comm%
\newcommand{\RD}[1]{##1}%
}%

\let\defi=\textsl

\SHOWCORRECTIONS

\begin{document}

    %
    %
\newbox\AYu
\setbox\AYu=\hbox{\large\textit{in memory of ~Alexander\,Vasil'ev}}

\title[Q.c.-extensions, Loewner chains, and the $\lambda$-Lemma]%
{Quasiconformal extensions, Loewner chains,\\ and the $\lambda$-Lemma\\[2.5ex] \hbox to\textwidth{\hss\usebox\AYu\hskip1.5em}
\vskip1.5ex\mbox{~}}


\REM{ \institute{P. Gumenyuk \at
              Institutt for matematikk og naturvitenskap, Universitetet i Stavanger,\\ N-4036 Stavanger, Norway\\
              Tel.: +47-5183-19-11\\
              \email{pavel.gumenyuk@uis.no}
           \and
           I. Prause \at
              Department of Mathematics and Statistics, P.O. Box 68 (Gustaf H\"allstr\"omin katu~2b), University of Helsinki FI-00014, Finland\\
              Tel.: +358-2-941-51435\\
              \email{istvan.prause@helsinki.fi} }}

\author[P. Gumenyuk]{Pavel Gumenyuk\,${}^\dag$}
\address{Institutt for matematikk og naturvitenskap, Universitetet i Stavanger, N-4036 Stavanger, Norway}
\email{pavel.gumenyuk@uis.no}
\thanks{${}^\dag$ Partially supported by {\it Ministerio de Econom\'\i{}a y Competitividad} (Spain)
project \hbox{MTM2015-63699-P}.}

\author[I. Prause]{Istv\'an Prause\,$^\ddag$}
\address{Department of Mathematics and Statistics, P.O. Box 68 (Gustaf H\"allstr\"omin katu 2b),
University of Helsinki FI-00014, Finland} \email{istvan.prause@helsinki.fi}
\thanks{$^\ddag$ Supported by Academy of Finland grants 1266182 and 1303765.}



    %
    %

\maketitle

\begin{abstract}
In 1972, J.~Becker [J. Reine Angew. Math. \textbf{255}] discovered a sufficient condition for quasiconformal extendibility of Loewner
chains. Many known conditions for quasiconformal extendibility of holomorphic functions in the unit disk can be deduced from his result.
We give a new proof of (a~generalization of) Becker's result based on Slodkowski's Extended $\lambda$-Lemma. Moreover, we
characterize all quasiconformal extensions produced by  Becker's (classical) construction and use that to obtain examples in which
Becker's extension is extremal (i.e. optimal in the sense of maximal dilatation) or, on the contrary, fails to be extremal.
\end{abstract}

\tableofcontents

\section{Introduction}
The study of conformal mappings of the unit disk~$\UD:=\{z\colon |z|<1\}$ admitting quasiconformal extensions is a classical topic in
Geometric Function Theory closely related with the Teichm\"uller Theory, see, e.g.
\cite{AksentevS:2002,Krushkal:Handbook,Lehto:Book,Takhtajan}. One of the main results of Loewner Theory states that the
class~$\clS$ of all conformal mappings $f\colon\UD\to\Complex$ normalized by $f(0)=0$, $f'(0)=0$, coincides with the range of the
map that takes each function $p(z,t)$, $z\in\UD$, $t\ge0$, measurable in~$t$, holomorphic in~$z$ with $\Re p>0$, and normalized
by~$p(0,t)=1$, to $f(z):=\lim_{t\to+\infty}e^t w(z,t)$, where $w$ is the unique  solution to the Loewner\,--\,Kufarev ODE $\di w/\di
t=-wp(w,t)$, $t\ge0$, $w(z,0)=z\in\UD$. At a first glance, this representation of~$\clS$ seems to be too complicated. Nevertheless, it
proved to be a very efficient tool in many problems of Complex Analysis, e.g., in extremal problems for conformal mappings,
see~\cite{FMSAYu} and references therein.

Given a subclass $\tilde \clS\subset\clS$, a natural problem arise: find a set of functions~$p$ that generates~$\tilde\clS$. For the
subclass $\clS_k$, $k\in(0,1)$, formed by all $f\in\clS$ admitting \hbox{$k$-q.c.} extensions $F:\ComplexE\to\ComplexE$ with
$F(\infty)=\infty$, see Sect.\,\ref{S_pre} for precise defintions, a partial answer was given in~1972 by Becker~\cite{Becker:1972}
who discovered a quite explicit construction of a $k$-q.c. extension of functions~$f\in\clS$ which are generated by~$p$ satisfying $
p(\UD,t)\subset U_k:=\big\{\zeta\colon|\zeta-1|\le k|\zeta+1|\big\}$ for a.e.~$t\ge0$. There are many indications that the
class~$\clS_k^B$ generated by such $p$'s does not coincide with~$\clS_k$. In Sect.\,\ref{S_BeckExt}, we characterize all
q.c.-extensions arising from Becker's construction, see Theorem~\ref{TH_Becker-converse}. As a corollary, we are able to prove
rigorously that~$\clS_k^B\neq \clS_k$, see Theorem~\ref{TH_NO} in Sect.\,\ref{S_NO}. Although Becker's condition is only sufficient
for $k$-q.c. extendibility, it seems to be worth for thorough investigation. Many well-known explicit sufficient conditions can be
deduced from Becker's result, see, e.g. \cite[\S\S5.3-5.4]{Becker:1980}, \cite{Hotta::Explicit,Hotta::ComplCoeff}. In fact, for certain
$f\in\clS_k$, Becker's extension has smallest maximal dilatation  among all q.c.-extensions of~$f$, see, e.g. Examples~\ref{EX_extrI}
and~\ref{EX_extrII} in Sect.\,\ref{S_NO}.
 Moreover, in~\cite{HottaGum::QC-chordal}, Becker's construction has been extended to the so-called chordal analogue  of the
Loewner\,--\,Kufarev equation~\cite[\S8]{BracciCD:evolutionI}. A bit surprisingly, but Becker's condition appears to be sufficient for
the $k$-q.c. extendibility in a much more general version of Loewner Theory developed
in~\cite{BracciCD:evolutionI,BracciCD:evolutionII,CDGloewner-chains,ABHK}. The proof of that fact has been recently obtained
in~\cite{Ikkei-pre}. In Sect.\,\ref{SS_withLambdaLemma}, we give a shorter proof based on the Extended $\lambda$-Lemma due
Slodkowski~\cite{Slodkowski:1991}.

In the next section, we give necessary preliminaries from quasiconformal mappings and Loewner Theory. Our main results are stated
and proved in Sect.\,\ref{SS_withLambdaLemma}\,--\,\ref{S_NO}. The paper is concluded with a brief discussion in
Sect.\,\ref{S_L-range} on an auxiliary question concerning the Loewner range, which constitutes also some independent interest for
Loewner Theory.

\section{Preliminaries}\label{S_pre}
There are several ways to define quasiconformality. One of the equivalent definitions is as follows.
\begin{definition}\label{DF_qc}
Let $k\in[0,1)$.  By a \textsl{$k$-quasiconformal map} (in short, \textsl{$k$-q.c. map}, or simply \textsl{q.c.-map} if we do not have
to specify~$k$) of a domain $U\subset\ComplexE$ we mean a homeomorphism $F:U\to\ComplexE$ in the Sobolev
class~$W^{1,2}_{\mathrm{loc}}$ such that $|\bar\partial F|\le k|\partial F|$ for a.e.~$z\in U$, where as usual $
\bar\partial:=\tfrac12(\tfrac\partial{\partial x}+i\tfrac\partial{\partial y})~\text{and}~
\partial:=\tfrac12(\tfrac\partial{\partial x}-i\tfrac\partial{\partial y}).
$
\end{definition}
Every $k$-q.c. map $F$ is a solution to the Beltrami equation
\begin{equation}\label{EQ_Beltrami}
\bar \partial F(z)=\mu_F(z)\,\partial F(z) \qquad \text{for a.e.~$z$,}
\end{equation}
where $\mu_F$ is a complex-valued measurable function satisfying $\esssup|\mu_F|\le k$. The map $F$ is conformal if $\mu_F$
vanishes almost everywhere.
\begin{definition}
The function~$\mu_F$ in~\eqref{EQ_Beltrami} is called the \textsl{Beltrami coefficient} or \textsl{complex dilatation} of~$F$; and
$\esssup|\mu_F|$  is called the \textsl{maximal dilatation}~of~$F$.
\end{definition}
Note that in this paper we use the ``$k$-small notation''. Often instead of~$k$, the deviation from conformality is measure by the
parameter $K:=(1+k)/(1-k)$. For further details on quasiconformal mappings, see,
e.g.~\cite{Ahlfors:2006,PDE-qc,LehtoVirtanen:1973,Lehto:Book}.

Criteria for q.c.-extendibility of univalent maps $f:\UD\to\Complex$ is a classical topic in Geometric Function Theory, see, e.g.
\cite{AksentevS:2002,Krushkal:Handbook,Lehto:Book,Takhtajan}.
\begin{definition}
Let $k\in[0,1)$. A holomorphic map $f:\UD\to\Complex$ is said to be \textsl{$k$-q.c. extendible} to~$\Complex$, or
to~$\ComplexE$, if there exists a $k$-q.c. map $F:\Complex\to\Complex$, or respectively, ${F:\ComplexE\to\ComplexE}$ such that
$F|_\UD=f$.
\end{definition}
Thanks to the removability property, see, e.g. \cite[Chapter~I, \S8.1]{LehtoVirtanen:1973}, $f$ is \hbox{$k$-q.c.} extendible to
$\Complex$ if and only if it admits a $k$-q.c. extension $F:\ComplexE\to\ComplexE$ with $F(\infty)=\infty$.  We will denote by
$\clS_k$ the class of all $k$-q.c. extendible~to~$\Complex$ holomorphic
 functions ${f:\UD\to\Complex}$ normalized by $f(0)=f'(0)-1=0$.

\medskip
Below we collect necessary basics of Loewner Theory
following~\cite{BracciCD:evolutionI,BracciCD:evolutionII,CDGloewner-chains,ABHK}. In~1972, Becker~\cite{Becker:1972} obtained
an explicit construction of q.c.-extensions based on so-called \textsl{Loewner chains}.
\begin{definition}[\protect{\cite[Definition~1.2]{CDGloewner-chains}}]
A \textsl{Loewner chain} in the unit disk~$\UD$ is a family~$(f_t)_{t\ge0}\subset\Hol(\UD,\Complex)$ satisfying the following
conditions:
\begin{itemize}
\item[LC1.]  $f_t$ is univalent in~$\UD$ for any~$t\ge0$,
\item[LC2.] $f_s(\UD)\subset f_t(\UD)$ whenever $t\ge s\ge0$,
\item[LC3.] for any compact $K\subset\UD$ there exists a non-negative locally integrable function $k_K$ on $[0,+\infty)$ such that
$$
\max_K\big|f_t-f_s\big|\le\int_s^t\!k_K(\xi)\,\di\xi \qquad\text{whenever $t\ge s\ge0$.}
$$
\end{itemize}
\end{definition}
\noindent Any Loewner chain~$(f_t)$ solves, in the sense of~\cite[Definition~2.1]{CDGduality}, the Loewner\,--\,Kufarev PDE
\begin{equation}\label{EQ_genLK-PDE}
\frac{\partial f_t}{\partial t}=-f'_t(z)G(z,t),\quad z\in\UD,~t\ge0,
\end{equation}
where $G$ is some \textsl{Herglotz vector field}, defined by~$(f_t)$ uniquely up to a null-set on the~$t$-axis. According to
\cite[Theorem~4.8]{BracciCD:evolutionI}, one of the two equivalent definitions of Herglotz vector fields is as follows.
\begin{definition}
A \textsl{Herglotz function} in~$\UD$ is a map $p:\UD\times[0,+\infty)\to\Complex$ satisfying the following conditions:
\begin{itemize}
\item[HF1.] for each $z\in\UD$, $p(z,\cdot)$ is locally integrable on~$[0,+\infty)$, and
\item[HF2.]  for a.e.~$t\ge0$, $p(\cdot,t)$ is holomorphic in~$\UD$ and $\Re p(\cdot,t)\ge0$.
\end{itemize}
A \textsl{Herglotz vector field} in~$\UD$ is a map~$G:\UD\times[0,+\infty)\to\Complex$ of the form
\begin{equation}\label{EQ_BP}
G(z,t)=\big(\tau(t)-z\big)\big(1-\overline{\tau(t)}z\big)\,p(z,t)\quad\text{for all~$z\in\UD$ and a.e.~$t\ge0$},
\end{equation}
where $\tau:[0,+\infty)\to\overline\UD$ is a measurable function and $p$ is a Herglotz function.
\end{definition}
\begin{remark}
It is known~\cite[Theorem~4.8]{BracciCD:evolutionI} that the Herglotz function $p$ in representation~\eqref{EQ_BP} is uniquely
defined by~$G$ up to a null-set on the $t$-axis.
\end{remark}

Furthermore, it is known that the so-called \textsl{evolution family $(\varphi_{s,t})_{t\ge s\ge 0}$, $\varphi_{s,t}:=f_t^{-1}\circ
f_s$, associated with~$(f_t)$} is the unique solution to the following initial values problem for the (generalized) Loewner\,--\,Kufarev
ODE:
\begin{equation}\label{EQ_genLK-ODE}
\frac{\di\varphi_{s,t}(z)}{\di t}=G\big(\varphi_{s,t}(z),t\big),\qquad t\ge s\ge0,~z\in\UD;\quad \varphi_{s,s}(z)=z.
\end{equation}
\begin{remark}
In general, the right-hand side of~\eqref{EQ_genLK-ODE} is discontinuous in~$t$. The equation is to be understood as
Carath\'eodory's first order ODE; see, e.g. \hbox{\cite[\S2]{CDGannulusI}} and references therein for basic theory of such ODEs.
\end{remark}

Evolution families can be defined independently, with no relation to Loewner chains.
\begin{definition}[\protect{\cite[Definition~3.1]{BracciCD:evolutionI}}]
An \textsl{evolution family} in~$\UD$ is a two-parameter family $(\varphi_{s,t})_{t\ge s\ge 0}$ satisfying the following conditions:
\begin{itemize}
\item[EF1.] $\varphi_{s,s}=\id_{\mathbb{D}}$ for any~$s\ge0$,

\item[EF2.] $\varphi_{s,t}=\varphi_{u,t}\circ\varphi_{s,u}$ whenever $t\ge u\ge s\ge0$,

\item[EF3.] for each $z\in\mathbb{D}$ there exists a non-negative locally integrable function $k_{z,T}$ on $[0,+\infty)$ such that
\[
|\varphi_{s,u}(z)-\varphi_{s,t}(z)|\le\int_{u}^{t}\!k_{z,T}(\xi)\,\di\xi\qquad\text{whenever $t\ge u\ge s\ge0$.}
\]
\end{itemize}
\end{definition}

Equation~\eqref{EQ_genLK-ODE} establishes a one-to-one correspondence between evolution families $(\varphi_{s,t})$ and Herglotz
vector fields~$G$, see \cite[Theorem~1.1]{BracciCD:evolutionI}. Moreover, given $(\varphi_{s,t})$ or~$G$, one can reconstruct the
corresponding Loewner chain~$(f_t)$, which turns out to be unique up to the post-composition with a conformal map,
see~\cite[Theorems 1.3,~1.6]{CDGloewner-chains} and~\cite{ABHK}.

\begin{definition}
By a \textsl{radial Loewner chain} we mean a Loewner chain~$(f_t)$ satisfying $f_t(0)=f_0(0)$ for all~$t\ge0$.
\end{definition}
\begin{remark}
Clearly, a Loewner chain $(f_t)$ is radial if and only if $G(0,t)=0$ for a.e.~$t\ge0$, i.e. if $\tau(t)=0$ and hence $G(z,t)=-z p(z,t)$ for
a.e.~$t\ge0$.
\end{remark}

\section{Becker's condition and q.c.-extensions of evolution families and general Loewner chains}\label{SS_withLambdaLemma}

The classical Loewner Theory developed in~\cite{Loewner:1923,Kufarev:1943,Pom:1965,Gutljanski}, see
also~\hbox{\cite[\S6.1]{Pom:1965}}, deals with radial Loewner chains~$(f_t)$ whose Herglotz functions $p$ are normalized by
$p(0,t)=1$ for a.e.~$t\ge0$. This normalization implies that the \textsl{Loewner range} $\cup_{t\ge0}f_t(\UD)$ coincides with~$\C$
and hence the Loewner chain~$(f_t)$ corresponding to a given Herglotz function~$p$ is defined uniquely up to linear transformations.
Therefore, many properties of~$(f_t)$, such as quasiconformal extendibility, are determined by the properties of~$p$.

Following Becker~\cite{Becker:1976}, \cite[\S5.1]{Becker:1980}, we replace the normalization ${p(0,t)=1}$ by a weaker
condition~$\int_0^{+\infty}\Re p(0,t)\di t=+\infty$, which still guarantees that $\cup_{t\ge0}f_t(\UD)=\C$. Becker discovered the
following remarkable result.
\begin{result}[\protect{\cite{Becker:1972,Becker:1976}}]\label{TH_Becker}
Let $k\in[0,1)$ and let $(f_t)$ be a radial Loewner chain whose Herglotz function~$p$ satisfies
\begin{equation}\label{EQ_Becker-condition}
  p(\UD,t) \subset U(k):=\left\{ w \in \C \colon \left|\frac{w-1}{w+1}\right| \le k\right\}\quad\text{for  a.e.~$t\ge0$}.
\end{equation}
Then for every~$t\ge0$, the function~$f_t$ admits a $k$-q.c. extension to~$\ComplexE$ that fixes~$\infty$. In particular, such an
extension for~$f_0$ is given by
\begin{equation}\label{EQ_BeckerExt}
F(\rho e^{i\theta}):=\left\{
 \begin{array}{ll}
  f_0(\rho e^{i\theta}), & \text{if~$0\le \rho<1$},\\[1ex]
  f_{\log \rho}(e^{i\theta}), & \text{if~$\rho\ge1$}.
 \end{array}
\right.
\end{equation}
\end{result}
Many sufficient conditions for q.c.-extendibility can be deduced from the above theorem, see e.g.~\cite[\S\S5.3-5.4]{Becker:1980}
and~\cite{Hotta::Explicit,Hotta::ComplCoeff}. A similar result for chordal Loewner chains, i.e. the Loewner chains associated with
Herglotz vector fields of the form $G(z,t)=(1-z)^2p(z,t)$, was obtained in~\cite{HottaGum::QC-chordal}. More generally,
Hotta~\cite{Ikkei-pre} showed that Becker's condition~\eqref{EQ_Becker-condition} is sufficient for q.c.-extendibility also in case of
Herglotz vector fields~\eqref{EQ_BP} with arbitrary measurable function~$\tau:[0,+\infty)\to\overline\UD$. Below we give a simpler
proof of this fact using the Extended $\lambda$-Lemma due to Slodkowski~\cite{Slodkowski:1991}.

\begin{theorem}\label{TH_q.c.-extension-via-lambda-lemma} Let $k\in[0,1)$ and let $(f_t)$ be a Loewner chain associated with a
Herglotz vector field~$G$ such that the Herglotz function $p$ in representation~\eqref{EQ_BP} satisfies Becker's
condition~\eqref{EQ_Becker-condition}. Then:
\begin{itemize}
 \item[(i)] $\cup_{t\ge0}f_t(\UD)=\Complex$;
 \item[(ii)]~all the elements of the Loewner chain~$(f_t)$ and of the evolution family~$(\varphi_{s,t})$ associated with~$G$ admit
     $k$-q.c. extensions to~$\ComplexE$.
\end{itemize}
\end{theorem}
\begin{remark}
By \cite[Theorem~1.6]{CDGloewner-chains}, for any Herglotz vector field~$G$ there exists a \textit{unique} associated Loewner
chain~$(f_t)$ such that $f_0(0)=f_0'(0)-1=0$ and $\cup_{t\ge0}f_t(\UD)$ is the whole complex plane or a disk centered at the
origin. A Loewner chain~$(f_t)$ satisfying these conditions is called \textsl{standard}.
\end{remark}
\noindent{\bf Proof of Theorem~\ref{TH_q.c.-extension-via-lambda-lemma}.} First we will replace~\eqref{EQ_Becker-condition} by a
weaker condition and prove~(ii) for the standard Loewner chain~$(f_t)$ associated with~$G$. Namely, assume that there exists locally
integrable ${a:[0,+\infty)\to\UH\cup i\R}$ such that
\begin{equation}\label{EQ_Becker-weaker}
p(\UD,t)\subset D_t\quad\text{for a.e. $t\ge0$},
\end{equation}
where $D_t$ is the closed hyperbolic disk in~$\UH$ of radius~$\tfrac12\log\tfrac{1+k}{1-k}$ centered at~$a(t)$ when
$a(t)\in\UH$, and $D_t:=\{a(t)\}$ when $a(t)\in i\R$. If $a(t)\equiv1$, then \eqref{EQ_Becker-weaker} becomes Becker's
condition~\eqref{EQ_Becker-condition}.

For all $t\in\mathcal Q:=\{t\ge0:\Re a(t)>0\}$ and $\lambda\in\UD$, set $p_\lambda(\cdot,t):=H_t\circ \phi_\lambda(\cdot,t)$,
where
$$\phi_\lambda(\cdot,t):=\frac{\lambda}{k}~H_t^{-1}\circ p(\cdot,t)\quad\text{and}\quad H_t(z):=\frac{1+z}{1-z}\,\Re
a(t)\,+\,i\Im a(t).$$ For $t\in[0,+\infty)\setminus\mathcal Q$, $p(\cdot,t)$ is a constant belonging to~$i\R$ and we set
$p_\lambda(\cdot,t):=p(\cdot,t)$ for all such~$t$ and all~$\lambda\in\UD$. Since $H_t^{-1}(D_t)=\{z:|z|\le k\}$ for
all~$t\in\mathcal Q$, condition~\eqref{EQ_Becker-weaker} implies that $|\phi_\lambda(z,t)|\le1$ for a.e.~$t\in\mathcal Q$ and
all~$z,\lambda\in\UD$. It follows that~$p_\lambda$ is a Herglotz function for any~$\lambda\in\UD$.   For each $\lambda\in\UD$, let
$(\varphi_{s,t}^\lambda)$ stand for the evolution family associated with the Herglotz vector field
$G_\lambda(z,t):=\big(\tau(t)-z\big)\big(1-\overline{\tau(t)}z\big)\,p_\lambda(z,t)$. Note that  $p_k=p$ and hence
$(\varphi_{s,t}^k)=(\varphi_{s,t}).$

By the very construction, $\UD\ni\lambda\mapsto p_\lambda(z,t)$ is holomorphic for all~$z\in\UD$ and a.e.~$t\ge0$. Moreover, for
any compact sets $K_1,K_2\subset\UD$,
 $$t~\mapsto\max_{(\lambda,z)\in K_1\times K_2}|p_\lambda(z,t)|$$ is locally integrable on~$[0,+\infty)$.
Following the standard arguments used in \cite[\S2]{CDGannulusI}, it is easy to conclude that
$\UD\ni\lambda\mapsto\varphi^\lambda_{s,t}(z)$ is holomorphic whenever $t\ge s\ge0$ and $z\in\UD$.

For any $s\ge0$ and $\lambda\in\UD$, the Schwarzian $S\varphi^{\lambda}_{s,t}$ of $\varphi_{s,t}^\lambda$ satisfies
$$
\frac{\di}{\di t} S\varphi^{\lambda}_{s,t}(z)=\big({\di}\varphi^{\lambda}_{s,t}(z)/{\di z}\big)^2\,
G_\lambda'''\big(\varphi^{\lambda}_{s,t}(z),t\big) \quad\text{a.e.~$t\ge s$.}$$  In particular, $G_0'''\equiv0$ and hence
$\varphi_{s,t}^0$'s are linear-fractional maps.

Now, fix some $s\ge0$ and $t\ge s$. Note that  $\varphi^\lambda_{s,t}:\UD\to\UD$ is univalent for any $\lambda\in\UD$. Therefore,
$(\lambda,z)\mapsto\psi_\lambda(z):={(\varphi_{s,t}^0)^{-1}\circ\varphi^\lambda_{s,t}}$ is a holomorphic motion of~$\UD$. By
the Extended \hbox{$\lambda$-Lemma}, see, e.g. \cite[Theorem~12.3.2 on p.\,298]{PDE-qc}, it extends to a holomorphic motion
$\UD\times\ComplexE\ni(\lambda,z)\mapsto\Psi_\lambda(z)\in\ComplexE$ of~$\ComplexE$ and moreover, for any~$\lambda\in\UD$,
the map $\Phi^\lambda_{s,t}$ is a $|\lambda|$-q.c. automorphism of~$\ComplexE$. In particular, for~$\lambda:=k$ we obtain a
$k$-q.c. extension of $\varphi_{s,t}=\varphi_{s,t}^k$ defined by the formula $\Phi_\lambda:=\varphi_{s,t}^0\circ\Psi_\lambda$.  To
prove $k$-q.c. extendibility of~$f_t$'s, it remains to use the explicit construction of the associated standard Loewner chain given in the
proof of~\cite[Theorem~1.6]{CDGloewner-chains} and apply \cite[Theorem~14.1 on p.\,148]{Schober:1975}.

Now let us assume that $p$ satisfies~\eqref{EQ_Becker-condition}. Then (i) holds by Theorem~\ref{TH_LoewnerRange}, which we will
prove in Sect.\,\ref{S_L-range}. As a consequence, according to \cite[Theorem~1.6]{CDGloewner-chains},  any two Loewner chains
associated with~$G$ differ by a linear map. Therefore, (ii)~holds for any Loewner chain~$(f_t)$ associated with~$G$.\proofbox

\begin{remark}
In contrast to Becker's classical result, the q.c.-extensions of the evolution family and of the standard Loewner chain~$(f_t)$ in
Theorem~\ref{TH_q.c.-extension-via-lambda-lemma} do not have to fix~$\infty$. For the special case of constant~$\tau$, the
linear-fractional maps $\varphi_{s,t}^0$ have a fixed point at~$\tau^*:=1/\overline\tau$ and hence before extending
to~$\ComplexE$, we may define the holomorphic motion $(\lambda,z)\mapsto\psi_\lambda(z)$ also at the point $\tau^*$  by
setting $\psi_\lambda(\tau^*):=\tau^*$ for all~$\lambda\in\UD$. As a result, in this case, $\varphi_{s,t}$'s admit $k$-q.c.
extensions to~$\ComplexE$ with a fixed point at~$\tau^*$.
\end{remark}

\begin{remark}
Another way  to apply the Extended $\lambda$-Lemma to the problem of q.c.-extendibility (not related to Loewner chains) was found
by Sugawa~\cite{SugawaPol}. The Loewner\,--\,Kufarev equation for q.c.-extendible functions  as an evolution in the universal
Teichm\"uller space was studied by Vasil'ev~\cite{Vasiliev1,Vasiliev2}.
\end{remark}

\section{Characterization of Becker's Extensions}\label{S_BeckExt}
The q.c.-extensions produced by Becker's construction form a proper subclass in the class of all q.c.-maps of $\ComplexE$ that are
conformal in~$\UD$ and have a fixed point at~$\infty$.  Below we give a comparatively simple characterization of Becker's extensions.
To be precise, we start by introducing the following definition.
\begin{definition}\label{DF_Becker-qc-ext}
A \defi{Becker extension} of a function $f\in\clS$ is a q.c.-map ${F:\ComplexE\to\ComplexE}$ with $F|_\UD=f$, $F(\infty)=\infty$ and
such that there exists a radial Loewner chain~$(f_t)$ satisfying the following conditions:
\begin{itemize}
\item[(i)] $f_0=f$;
\item[(ii)] for any $t\ge0$, the function~$f_t$ extends continuously to~$\partial\UD$, with $f_t(\zeta)=F(e^t\zeta)$ for
    all~$\zeta\in\partial\UD$.
\end{itemize}
If exists, this radial Loewner chain $(f_t)$  is clearly unique; in what follows it will be referred to as the Loewner chain
\defi{associated} with the Becker extension~$F$.
\end{definition}
Denote by $\mathcal H^\infty(\UD)$ the Hardy class of all bounded holomorphic functions in~$\UD$. We say that a complex-valued
function $\nu\in L^{\infty}(\partial\UD)$ \defi{represents the boundary values} of a function $\varphi\in\mathcal H^\infty(\UD)$, if
$\lim_{r\to1-}\varphi(r\zeta)=\nu(\zeta)$ for a.e.~${\zeta\in\partial\UD}$.
\begin{theorem}\label{TH_Becker-converse}
Let $k\in[0,1)$ and let $F:\ComplexE\to\ComplexE$, $F(\infty)=\infty$, be a $k$-q.c. extension of some function $f\in\clS$. The
following assertions hold:\vskip1.5ex
\begin{itemize}
\item[(I)] $F$ is a Becker extension of~$f$ if and only if the complex dilatation $\mu_F$ of~$F$ obeys the following property: for
    a.e. $\rho>1$ the map $\UC\ni\zeta\mapsto\mu_F(\rho\zeta)$ represents boundary values of some $\varphi_\rho\in\mathcal
    H^\infty(\UD)$ with $\varphi_\rho(0)=\varphi_\rho'(0)=0$.\\
\item[(II)] If $F$ is a Becker extension of~$f$, then the Loewner chain~$(f_t)$ associated with~$F$ satisfies the
    Loewner\,--\,Kufarev equation
$$
 \frac{\partial f_t(z)}{\partial t}=zf_t'(z)p(z,t),\quad t\ge0,~z\in\UD,
$$
where $p(z,t):=(1+\varphi_{e^t}(z)/z^2)/(1-\varphi_{e^t}(z)/z^2)$ for all~$z\in\UD$ and a.e.\,$t\ge0$. In~particular, $p$
satisfies Becker's condition~\eqref{EQ_Becker-condition}.
\end{itemize}
\end{theorem}
\begin{remark}
Statement~(I) of the above theorem can be rewritten as follows: $F$~is a Becker extension if and only if  the Fourier coefficients
$a_n(\rho):=\tfrac1{2\pi}\int_0^{2\pi}\!e^{-in\theta}\mu_F(\rho e^{i\theta})\,\di\theta$ vanish for a.e.~$\rho>1$ and all integer
$n\le 1$.
\end{remark}

In the proof, we will need the following lemma.
\begin{lemma}\label{LM_conv}
Let $F$ be a Becker extension of some $f\in\clS$. Then for a.e.\,$t\ge0$, the Loewner chain~$(f_t)$ associated with~$F$ satisfies:
\begin{align}
\label{EQ_f-prime}
\lim_{r\to1-}ie^{i\theta}f'_t(re^{i\theta})&=\frac{\partial F(e^{t+i\theta})}{\partial\theta}&&\text{for
a.e.\,$\theta\in[0,2\pi]$,}\\
\label{EQ_f-dot}
 \lim_{r\to1-}\,\,\frac{\partial f_t(re^{i\theta})}{\partial t}\,\,&=\frac{\partial F(e^{t+i\theta})}{\partial t}&&\text{for
a.e.\,$\theta\in[0,2\pi]$}.
\end{align}
\end{lemma}
\begin{proof}
For a.e.~$t\ge0$ the map~$\Real\ni\theta\to F(e^{t+i\theta})$ is absolutely continuous and hence by a theorem of F.\,Riesz, see,
e.g. \cite[Theorem~1 in~\S{}IX.5]{Goluzin}, equality~\eqref{EQ_f-prime} takes place.  To prove~\eqref{EQ_f-dot} fix for a moment
some $t\ge0$ and write
\begin{equation}\label{EQ_Poisson-diff}
\frac{f_{t+h}(z)-f_t(z)}{h}=\frac{1}{2\pi}\int\limits_{0}^{2\pi}\Poiss(e^{-i\theta}z)\,
\frac{F(e^{t+h+i\theta})-F(e^{t+i\theta})}{h}\,\,\di\theta
\end{equation}
for all $z\in\UD$ and all $h\ge -t$, where $\Poiss(z):=\Re \big((1+z)/(1-z)\big)$ is the Poisson kernel for~$\UD$. According to
\cite[Theorem~3.5.3 on p.\,66]{PDE-qc} and \cite[Corollary~3.4.7 on p.\,62]{PDE-qc}, for any $\varepsilon>0$ the function
$$L_\varepsilon(z):=\sup\left\{|F(w)-F(z)|/|w-z|: 0<|w-z|<\varepsilon\right\}$$ is locally integrable in~$\C$. It follows that
$\theta\mapsto \sup\{{(F(e^{t+h+i\theta})-F(e^{t+i\theta}))/h}: {h\in\Real,\,0<|h|<\varepsilon}\}$ is integrable on~$[0,2\pi]$ for
a.e.~$t\ge0$. Hence for all such~$t$'s we can apply Lebesgue's dominated convergence theorem to pass in~\eqref{EQ_Poisson-diff} to
the limit as~$h\to0$. As a result we get
$$
\frac{\partial f_t(z)}{\partial t}=\frac{1}{2\pi}\int\limits_{0}^{2\pi}\Poiss(e^{-i\theta}z)\, \frac{\partial F(e^{t+i\theta})
}{\partial t}\,\di\theta
$$
for all~$z\in\UD$ and a.e.~$t\ge0$. Equality~\eqref{EQ_f-dot} follows now by properties of the Poisson integral, see, e.g.
\cite[Corollary~1 in~\S{}IX.1]{Goluzin}.
\end{proof}
\noindent{\bf Proof of Theorem~\ref{TH_Becker-converse}.} Suppose that $F$ is a Becker extension of some~${f\in\clS}$ and
let~$(f_t)$ be the associated Loewner chain. Furthermore, let $p(z,t):=(\tfrac{\partial }{\partial t}f_t(z))/(z f'_t(z))$ be the Herglotz
function of~$(f_t)$. Since  ${\Re p(\cdot,t)\ge0}$,
$$\varphi_{e^t}(z):=\frac{p(z,t)-1}{p(z,t)+1}\,z^2=\frac{\partial f_t(z)/\partial t+i \big(izf'_t(z)\big)}{\partial f_t(z)/\partial t-i
\big(izf'_t(z)\big)}\,z^2$$ is a holomorphic bounded function in~$\UD$ for a.e.~$t\ge0$, with zero of the second order at~$z=0$.
 The Jacobian of $F$, $J_F(z)=|\partial
F|^2-|\bar\partial F|^2$ is positive for a.e.~$z$, see, e.g. \cite[Corollary~3.7.6 on p.\,75]{PDE-qc}. Therefore,
$\big(\tfrac{\partial}{\partial t}-i\tfrac{\partial}{\partial \theta}) F(e^{t+i\theta})\neq0$ for a.e. $(t,\theta)\in\Real\times[0,2\pi]$.
Taking this into account, we immediately deduce from Lemma~\ref{LM_conv} that for a.e.~$t\ge0$ and a.e.~$\theta\in[0,2\pi]$,
$\lim_{r\to1-}\varphi_{e^t}(re^{i\theta})=\mu_F(e^{t+i\theta})$. Since every bounded harmonic function can be recovered from its
radial limits on~$\UC$ by means of  the Poisson integral, see, e.g.\,\cite[p.\,38]{Hoffman:book}, we conclude, in particular, that
$\varphi_{e^{t}}(\UD)\subset k\UD$ for a.e.~$t\ge0$ and hence~$p$ satisfies Becker's condition~\eqref{EQ_Becker-condition}.  The
above argument proves~(II) and the necessity part of~(I).

Now suppose that for a.e.~$\rho>1$, the function
$\UC\ni\zeta\mapsto\mu_F(\rho\zeta)$ represents the boundary values of some ${\varphi_\rho\in\mathcal H^\infty(\UD)}$ with
${\varphi_\rho(0)=\varphi_\rho'(0)=0}$. We have to show that $F$ is a Becker extension. Recall that $\varphi_\rho$ can be recovered
from its radial limits on~$\partial\UD$ by means of the Poisson integral. Therefore, $\sup_{z\in\UD}|\varphi_\rho(z)| \le k$ for a.e.
$\rho>1$ and $\rho\mapsto\varphi_\rho(z)$ is measurable for each $z\in\UD$ by Fubini's Theorem applied
to~$(\rho,\theta)\mapsto{\mu_F(\rho e^{i\theta})(1+e^{-i\theta}z)/(1-e^{-i\theta}z)}$. It follows that the formula
$p(z,t):=\big(1+\varphi_{e^t}(z)/z^2\big)/\big(1-\varphi_{e^t}(z)/z^2\big)$ defines a Herglotz function
satisfying~\eqref{EQ_Becker-condition}.  By Becker's Theorem~\ref{TH_Becker}, the radial Loewner chain~$(f_t)$ associated
with~$p$ generates a $k$-q.c. extension $\Phi:\ComplexE\to\ComplexE$ of~$f_0$ with $\Phi(e^{t+i\theta})=f_t(e^{i\theta})$,
$t\ge0$, $\theta\in\Real$, and $\Phi(\infty)=\infty$. By definition, $\Phi$ is a Becker extension. Hence to complete the proof it
remains to check that  $F=\Phi$.
 Fix for a moment some~$r\in(0,1)$. Following  Becker's proof, we consider the $k$-q.c. map~$\Phi_r$ defined by equalities
$\Phi_r(z):=f_0(rz)/r$ for all~$z\in\UD$ and $\Phi_r(e^{t+i\theta}):=f_t(re^{i\theta})/r$ for all~$t\ge0$ and
all~$\theta\in[0,2\pi]$. Then $\mu_{\Phi_r}|_\UD=0$ and for all $\theta\in[0,2\pi]$ and a.e.~$t\ge0$ we have
 $$
\mu_{\Phi_r}(e^{t+i\theta})=\frac{\partial f_{t}(re^{i\theta})/\partial t\,-\,zf'_{t}(re^{i\theta})}{\partial
f_{t}(re^{i\theta})/\partial
t\,+\,zf'_{t}(re^{i\theta})}\,e^{2i\theta}=\frac{p(re^{i\theta},t)-1}{p(re^{i\theta},t)+1}\,e^{2i\theta}=
\frac{\varphi_{e^t}(re^{i\theta})}{r^2}.
$$
The radial limit of the r.h.s. exists for a.e.~$\theta\in[0,2\pi]$ and equals to~$\mu_F(e^{t+i\theta})$. It follows that
$\mu_{\Phi_r}\to\mu_F$ as~${r\to1^-}$ a.e. in~$\C$. Note also that $\Phi_r$ satisfies the same normalization as~$F$, i.e.
$\Phi_r(0)=0$, $\Phi'_r(0)=1$, and $\Phi_r(\infty)=\infty$.  Therefore,  according to~\cite[Lemma~5.3.5 on p.\,171]{PDE-qc},
$\Phi_r\to F$ in~$\ComplexE$ as $r\to1^-$. On the other hand, it follows easily from the construction that $\Phi_r\to\Phi$
as~$r\to1^-$. This completes the proof. \proofbox
\begin{remark}
It is interesting to compare Becker's explicit construction with the machinery used in the proof of
Theorem~\ref{TH_q.c.-extension-via-lambda-lemma} when applied to radial Loewner chains. Suppose that the Herglotz function $p$
associated with $(f_t)$ satisfies Becker's condition~\eqref{EQ_Becker-condition} and let $F$ be the Becker extension of~$f_0$ defined
by equality~\eqref{EQ_BeckerExt} in Theorem~\ref{TH_Becker}. Consider the holomorphic motion $(\lambda,z)\mapsto
F_\lambda(z)$, where $F_\lambda$ is the unique $|\lambda|$-q.c. automorphism of~$\ComplexE$ satisfying the Beltrami equation
$\bar \partial F_\lambda(z)=k^{-1}\lambda\mu_F(z)\partial F_\lambda(z)$ with the normalization $F_\lambda(\infty)=\infty$,
$F_\lambda(0)=0$, $F'_\lambda(0)=1$. Then according to Theorem~\ref{TH_Becker-converse}, for any~$\lambda\in\UD$,
$F_\lambda$ is a Becker extension of some~$f^\lambda\in\clS$. Let $(f_t^\lambda)$ be its associated classical Loewner chain. It is
easy to see from the proof of Theorem~\ref{TH_Becker-converse} that the Herglotz function of~$(f_t^\lambda)$ coincides with
$p_\lambda$ defined in the proof of~Theorem~\ref{TH_q.c.-extension-via-lambda-lemma}. Therefore, in the classical setting,
Becker's explicit q.c.-extension is one of the q.c.-extensions that may be obtained via our implicit construction based on the
$\lambda$-Lemma.
\end{remark}

\section{Examples and remarks}\label{S_NO}
For $k\in[0,1)$, denote by $\clS_k$ the set of all $f\in\clS$ admitting a $k$-q.c. extension $F\colon\ComplexE\to\ComplexE$
with~$F(\infty)=\infty$. We would like to get some idea about relation between the classes $\clS_k$  and the classes $\clS_k^B$
formed by all~$f\in\clS$ admitting a $k$-q.c. Becker extension. It is difficult to believe that $\clS_k^B$ coincides with $\clS_k$
or constitutes a ``large'' part of it. However,
$$
\clS_k\subset\clS_{3k}^B \qquad\text{for any $k\in[0,\tfrac13)$.}
$$
The above inclusion follows immediately from the following two facts. On the one hand,  for any $f\in\clS_k$, the Schwarzian~$S_f$ of
$f$ satisfies the inequality ${(1-|z|^2)^2 |S_f(z)|\le 6k}$ for all $z\in\UD$, see \cite[Satz~$3^*$]{Kuehnau:1969} or
\cite[Corollary~2]{Lehto:1971}. On the other hand, the condition ${(1-|z|^2)^2 |S_f(z)|\le 2k'}$ for all $z\in\UD$ is sufficient for
$k'$-q.c. extendibility \cite{Ahlfors:1974,Ahlfors-Weil} and, moreover, implies the existence of a $k'$-q.c. Becker extension
\cite[Satz~4.2]{Becker:1972}, \cite[p.\,62--68]{Becker:1980}.

 Moreover, in certain examples, a Becker extension is the best possible in the sense of the maximal dilatation.

\begin{example}\label{EX_extrI}
Fix $k\in[0,1)$. Let $f_1(z):=z/(1-kz)^2$ and $f_2(z):=z/(1-kz^2)$. It  follows readily from \cite[Satz~$3^*$]{Kuehnau:1969},
and from similar results in \cite[Corollaries~1 and~3]{Lehto:1971}, that $f_j$, $j=1,2$, have unique $k$-q.c. extensions
$F_j:\ComplexE\to\ComplexE$ with $F_j(\infty)=\infty$. In fact, $F_j$'s are Becker extensions associated with the Loewner chains
$f^j_t:=e^t f_j$, $t\ge0$, $j=1,2$.
\end{example}

\begin{example}\label{EX_extrII}
Let $\sigma\in(0,2)$ and consider $f_\sigma\in\clS$ given by~$f_\sigma:=\sigma^{-1}H^{-1}\circ g_\sigma\circ H$, where
$g_\sigma(\zeta):=\zeta^\sigma$, $\Re\zeta>0$, $g_\sigma(1)=1$, and $H(z):=(1+z)/(1-z)$. Let $F$ be a q.c.-extension
of~$f_\sigma$ to~$\ComplexE$ with~$F(\infty)=\infty$. Define $\psi:=(2-\sigma)^{-1}h^{-1}\circ (\sigma F)\circ h$, where
$h(w):=H^{-1}(-e^w)$ is a conformal mapping of $\Pi:={\{w\colon |\Im w|<\pi/2\}}$ onto $\ComplexE\setminus\overline\UD$. Then
$\psi$ is a q.c.-automorphism of~$\Pi$  continuously extendible to $\partial \Pi$ with
\begin{equation}\label{EQ_boundary-conditions}
\psi\big(x\pm \frac{i\pi}2\big)=\frac{\sigma x}{2-\sigma}\pm \frac{i\pi} 2\quad\text{ for all~$x\in\Real$.}
 \end{equation}
 Conversely, any q.c.-automorphism of~$\Pi$ satisfying~\eqref{EQ_boundary-conditions} defines a q.c.-extension~$F$ of~$f$
to~$\ComplexE$, with $\mu_F(z)=\mu_\psi(h^{-1}(z))(z^2-1)/(\overline z\,^2-1)$, $|z|>1$.
 Let $k:=\esssup|\mu_\psi|$ be the maximal dilatation of~$\psi$. To estimate $k$ from below, consider the $k$-q.c. automorphism
of~$\UD$ defined by $\varphi:=H^{-1}\circ\mathrm{exp}\circ\psi\circ \mathrm{log}\circ H$, where $\log$ stands for the branch of
the logarithm that maps the right half-plane onto~$\Pi$. Note that $\varphi$ and $\varphi^{-1}$, as well as their homeomorphic
extensions to~$\overline\UD$, must be ${(1-k)/(1+k)}$-H\"older continuous; see, e.g. \cite[p.\,30]{Ahlfors:2006}. It follows that
\begin{multline*}
 \inf_{a\in\Real}\sup_{w\in P_a}\frac{\Re\psi(w)}{\Re w}\le K,\quad\sup_{a\in\Real}\inf_{w\in P_a}\frac{\Re\psi(w)}{\Re w}\ge
\frac 1K,\\\text{where}\quad P_a:=\big\{w\colon \Re w\ge a,~|\Im w|\le\frac\pi2\big\}.
\end{multline*}
Using~\eqref{EQ_boundary-conditions}, we get $k\ge|\sigma-1|$. Equality occurs when $\psi=\psi_0$, $\psi_0(x+iy)=\sigma
x/(2-\sigma)+iy$ for all~$x\in\Real$, $y\in(-\pi/2,\pi/2)$. Moreover,  $\psi_0$ is the only extremal q.c.-map in this case; see, e.g.
\cite[Example~1.4.2 on p.\,85]{Reich:2002}.

Let $F_0$ be the q.c.-extension of $f_\sigma$ corresponding to the automorphism~$\psi_0$. Then $F_0$ is the unique
$|\sigma-1|$-q.c. extension of~$f_\sigma$ to~$\ComplexE$. A simple computation shows that $F_0(\infty)=\infty$ and that
$\mu_{F_0}(z)=(\sigma-1)(z^2-1)/(\overline z\,^2-1)$ for all~$z\in\Complex\setminus\overline\UD$. By
Theorem~\ref{TH_Becker-converse}, $F_0$ is a Becker extension of~$f_\sigma$.
\end{example}
\begin{remark}\label{RM_Teichmuller-maps}
The Becker extensions  $F_0$ and $F_2$ in the above two examples are uniquely extremal Teichm\"uller mappings of
$\ComplexE\setminus\overline\UD$ with infinite and finite norm, respectively; see, e.g. \cite{ReichStrebel} or~\cite{Reich:2002},
for the terminology and related results. The Becker extension $F_1$ is a Teichm\"uller mapping
of~$\Complex\setminus\overline\UD$, but not of~$\ComplexE\setminus\overline\UD$, and it is not extremal without the
condition~$F(\infty)=\infty$.
\end{remark}
 On the one hand, Examples~\ref{EX_extrI} and~\ref{EX_extrII} along with Remark~\ref{RM_Teichmuller-maps} indicate that
$\clS_k^B$ should be an ``important'' part of~$\clS$. In particular, one can construct a series of similar examples, e.g., by
considering ${f_n(z):=\big(f_1(z^n)\big)^{1/n}}={z/(1-kz^n)^{2/n}}$, $n=2,3,\ldots$ The Loewner chain $f^n_t:=e^t f_n$,
$t\ge0$, defines a $k$-q.c. Becker extension $F_n$ of the function~$f_n$, with ${\mu_{F_n}(z)=k\overline{\varphi(z)}/|\varphi(z)|}$
for all~$z\in\Complex\setminus\overline\UD$, where $\varphi(z):=-1/z^{n+2}$. Therefore, $F_n$ is a Teichm\"uller map
of~$\ComplexE\setminus\overline\UD$ with finite norm. It follows~\cite{Strebel} that $F_n$ is the unique q.c.-extension of~$f_n$ for
which the maximal dilatation has the least possible value.  On the other hand, the same idea allows us to construct, in an implicit way,
many functions $f\in\clS_k$ not belonging to~$f\in\clS_k^B$.
\begin{theorem}\label{TH_NO}
For any $k\in(0,1)$, $\clS_k^B\neq\clS_k$. In particular, if $\varphi$ is a holomorphic function in~$\C\setminus\overline\UD$ with
finite norm $\|\varphi\|_1:=\iint_{\Complex\setminus\overline{\UD}}|\varphi(z)|\,\di x\di y$, and if there exists $\rho>1$ such that
$\UC\ni\zeta\mapsto\zeta^{-2}\overline{\varphi(\rho\zeta)}/|\varphi(\rho\zeta)|$ does not admit a holomorphic extension to~$\UD$,
then $F|_\UD\in\clS_k\setminus\clS_k^B$, where $F$ stands for the unique solution to the Beltrami equation $\bar\partial F=\mu\,
\partial F$, $\mu:=k\overline\varphi/|\varphi|$ in $\Complex\setminus\overline\UD$, $\mu\equiv0$ in~$\UD$, normalized by $F(0)=0$,
$F'(0)=1$, $F(\infty)=\infty$.
\end{theorem}
\begin{proof}
By \cite[Theorem~4]{Strebel:1978}, $F$ is the unique $k$-q.c. extension of~$f:=F|_\UD$, but it is not a Becker extension by
Theorem~\ref{TH_Becker-converse}.
\end{proof}

At the end of this section, let us recall that Theorem~\ref{TH_q.c.-extension-via-lambda-lemma} states also q.c.-extendibility of the
evolution family. Fix $k\in[0,1)$ and denote by $\mathcal U_k$ the union of all evolution families generated by Herglotz vector fields
given by~\eqref{EQ_BP} with $p$ satisfying condition~\eqref{EQ_Becker-weaker}.
\begin{remark}
An interesting fact about $\mathcal U_k$ is that on the one hand, it is closed w.r.t. taking compositions, i.e. $\varphi\circ\psi\in
\mathcal U_k$ for any $\varphi,\psi\in\mathcal U_k$, but on the other hand, each $\varphi\in\mathcal U$ admits a
q.c.-extension~$\Phi$ to~$\ComplexE$ \textit{with the same bound for the maximal dilatation}: $\esssup|\mu_\Phi|\le k$, although
the composition of $k$-q.c. extensions of $\varphi$ and $\psi$, clearly, does not need to be a~$k$-q.c.~map.
\end{remark}

\section{A question concerning the Loewner range}\label{S_L-range}
In order to complete the proof of Theorem~\ref{TH_q.c.-extension-via-lambda-lemma}, we have to show that under its hypothesis,
the Loewner range $L[(f_t)]:=\cup_{t\ge0}f_t(\UD)$ is the whole complex plane.
\begin{theorem}\label{TH_LoewnerRange}
Let $p$ be a Herglotz function such that $p(\UD,t)\subset K$ for a.e.~$t\ge0$ and some compact set $K\subset\UH$. Then for any
measurable  ${\tau\colon [0,+\infty)\to\overline\UD}$, the Loewner chains~$(f_t)$ associated with the Herglotz vector field
$G(z,t):=(\tau(t)-z)(1-\overline{\tau(t)}z)p(z,t)$ satisfy $L[(f_t)]=\Complex$.
\end{theorem}
Before giving the proof of above theorem, let us place some remarks. If $\tau$ is a constant function with the value in~$\UD$, then it
is a simple exercise to show that $L[(f_t)]=\Complex$ if and only if $\int_0^{\infty}\Re p(0,t)\,\di t=+\infty$. In case of
constant~$\tau$ with the value on~$\UC$, a sufficient condition is that
\begin{equation}\label{EQ_sufficient-chordal}
C_1< \Re p(z,t) < C_2\quad\text{for all~$z\in\UD$ and a.e.~$t\ge0$}
\end{equation}
with some positive constants $C_1$ and $C_2$, see \cite[Proposition~3.7]{HottaGum::QC-chordal}. However,
Example~\ref{EX_essential} given after the proof of Theorem~\ref{TH_LoewnerRange}  shows that for arbitrary measurable functions
$\tau$, condition~\eqref{EQ_sufficient-chordal} does not imply that $L[(f_t)]=\Complex$.

\medskip\noindent{\bf Proof of Theorem~\ref{TH_LoewnerRange}.} Denote by $(\varphi_{s,t})$ the evolution family associated with
the Herglotz vector field~$G$. Let $\psi_{s,t}:=h_t\circ\varphi_{s,t}\circ h_s^{-1}$, where
$$
 h_t(w):=\frac{w-a(t)}{1-\overline{a(t)}w},\quad
a(t):=\varphi_{0,t}(0);\quad w\in\UD,~t\ge s\ge0.
$$
By \cite[Lemma~2.8]{CDGloewner-chains}, $(\psi_{s,t})$ is an evolution family. Thanks to \cite[Theorem~1.6]{CDGloewner-chains},
it is sufficient to show that $|\psi_{0,t}'(0)|\to0$ as $t\to+\infty$.  Denote by $G_0$ the Herglotz vector field of~$(\psi_{s,t})$. Since
$\psi_{s,t}(0)=0$ whenever $0\le s\le t$, we have $G_0(z,t)=-z q(z,t)$ for all~$z\in\UD$ and a.e.~$t\ge0$, where $q$ is a Herglotz
function. From~\eqref{EQ_genLK-ODE} we find that $\Re q(0,t)=$
\begin{equation*}
=-\Re G_0'(0,t)=\frac{1-|a(t)|^2}{\big|1+\overline{a(t)}\kappa(t)\big|^2}\,\,\Re\!\Big[(1+|\kappa(t)|^2)
p_1(0,t)-\kappa(t)p_1'(0,t)\Big],
\end{equation*}
for a.e.~$t\ge0$, where $p_1(z,t):=p(h_t^{-1}(z),t)$ and $\kappa(t):=h_t(\tau(t))$. Using the fact that holomorphic maps are
non-expansive w.r.t. the hyperbolic metric,  in the same way as in the proof of \cite[Proposition~3.7]{HottaGum::QC-chordal} we see
that $|p_1'(0,t)|\le 2\nu\Re p_1(0,t)$ for a.e.~$t\ge0$ and some constant~$\nu\in(0,1)$ depending only on the compact set~$K$.
Therefore, for a.e. $t\ge0$,
\begin{equation*}\label{EQ_RE-est}
\Re\!\Big[(1+|\kappa(t)|^2) p_1(0,t)-\kappa(t)p_1'(0,t)\Big]\ge (1-\nu)(1+|\kappa(t)|^2) \Re p_1(0,t).
\end{equation*}

To show that  $\lim\limits_{t\to+\infty}\log|\psi_{0,t}'(0)|=-\int_0^{+\infty}\Re q(0,t)\,\di t \
  =-\infty$, it remains to notice that
\begin{eqnarray*}
\displaystyle \frac{\di|a(t)|^2\!/\di t}{2}=\Re\!\Big[\overline{a(t)}\,G(a(t),t)\Big]\!\!&\le\,
 \displaystyle x\,\big|p(a(t),t)\big|\,\big(1-|a(t)|^2\big)-x^2\,\Re  p(a(t),t)\\
&\le\displaystyle\frac{\big|p(a(t),t)\big|^2}{4\Re p(a(t),t)}\,\big(1-|a(t)|^2\big)^2~~\text{~for a.e. $t\ge0$},
\end{eqnarray*}
where $x:=\big|1-\overline{\tau(t)}\,a(t)\big|$, and hence $\int_0^{+\infty}(1-|a(t)|^2)\,\di t=+\infty$. \proofbox

\begin{example}\label{EX_essential}
Let $\rho:[0,+\infty)\to[0,1)$ be a locally absolutely continuous function with $\rho(t)=0$ if and only if~$t=0$ and such that
$\int_0^1\!\di t/\rho(t)<+\infty$. Consider the evolution family $(\varphi_{s,t})$ associated with the Herglotz vector field~$G$ given
by~\eqref{EQ_BP} with
\begin{eqnarray*}
 p(z,t):=1-i\rho'(t)\frac{(1+\rho(t)^2)}{(1-\rho(t)^2)^2}~\text{~and~}~\tau(t):=&\displaystyle
  ie^{i\theta(t)}\frac{(1-i\rho(t))^2}{1+\rho(t)^2},\\[1ex]
\text{where~}~\theta(t):=&\displaystyle\int_0^t\!\frac{(1-\rho(s)^2)^2}{1+\rho(s)^2}\,\frac{\di s}{\rho(s)},
\end{eqnarray*}
for all~$z\in\UD$ and $t\ge0$. It is easy to check that~$a(t):=\varphi_{0,t}(0)=\rho(t)e^{i\theta(t)}$.

Calculations in the proof of Theorem~\ref{TH_LoewnerRange} show that $\psi_{0,t}'(0)\to0$ as ${t\to+\infty}$ if and only if
${\int_0^{+\infty}(1-\rho(t)^2)\,\di t=+\infty}$. It follows that the hypothesis of \hbox{Theorem~\ref{TH_LoewnerRange}} cannot
be replaced by the weaker condition~\eqref{EQ_sufficient-chordal}.
\end{example}

\bibliographystyle{plain}

\end{document}